\documentclass[11pt,leqno]{amsart}
\usepackage{epsfig}
\usepackage{amssymb}
\usepackage{amscd}
\usepackage[matrix,arrow]{xy}

\newtheorem{theo}{Theorem}[section]

\newtheorem{defi}[theo]{Definition}
\newtheorem{prop}[theo]{Proposition}

\numberwithin{equation}{section}

\def\pre-tr{\operatorname{pre-tr}}

\newcommand{\bbA}{{\mathbb A}}
\newcommand{\bbC}{{\mathbb C}}

\newcommand{\bbL}{{\mathbb L}}

\newcommand{\bbP}{{\mathbb P}}

\newcommand{\bbF}{{\mathbb F}}

\newcommand{\cV}{{\mathcal V}}

\newcommand{\cH}{{\mathcal H}}

\newcommand{\Symm}{\operatorname{Sym}}

\newcommand{\sign}{\operatorname{sign}}

\author{Michael Larsen \and Valery A.~Lunts}

\address{
  Department of Mathematics\\
  Indiana University\\
  Rawles Hall\\
  831 East 3rd Street\\
  Bloomington, IN 47405\\
  USA
}

\email{mjlarsen@indiana.edu}

\email{vlunts@indiana.edu}

\thanks{}

\title[Rationality of motivic zeta function and cut-and-paste problem]
{Rationality of motivic zeta function and cut-and-paste problem}

\begin{document}

\begin{abstract} Assuming the positive solution to the Cut-and-paste
problem we prove that the motivic zeta function remains irrational
after inverting $\bbL$.
\end{abstract}

\maketitle

\thanks{}

\section{Introduction}
Fix a field $\bbF$ and let $K_0[\cV _{\bbF}]$ denote the Grothendieck ring of
varieties over $\bbF$. That is $K_0[\cV _{\bbF}]$ is the abelian group
which is generated by isomorphism classes of $\bbF$-varieties with relations
$$[X]=[Y]+[X\backslash Y]$$
if $Y\subset X$ is a closed subvariety. The product in $K_0[\cV _{\bbF}]$ is
defined as
$$[X]\cdot [Y]=[X\times _{\bbF}Y]$$
In \cite{LaLu1} we have asked the following question:

\medskip

\noindent{\it Cut-and-paste problem}.
Let $Z_1,...,Z_k;W_1,...W_l$ be $\bbF$-varieties and consider the disjoint unions
$X=\coprod Z_i$ and $Y=\coprod W_j$. Suppose that $[X]=[Y]$. Is it possible to decompose $X$ and $Y$ into locally closed
subvarieties
$$X=\coprod _{i=1}^k X_i,\quad Y=\coprod _{i=1}^k Y_i$$
such that for each $i$ the varieties $X_i$ and $Y_i$ are isomorphic?

\medskip

Some positive results for this problem are obtained in the paper \cite{LiSeb}.
They prove
that the solution to the problem is positive (in characteristic zero) if 1)
$\dim X\leq 1$, 2) $X$ is a smooth connected projective surface, 3) $X$ contains
only finite many rational curves.

In this note we want to relate the Cut-and-paste problem to the question of
rationality of the motivic zeta function
$$\zeta _X(t)=\sum _{n=0}^\infty [\Symm ^nX]t^n \in K_0[\cV _{\bbF}][[t]]$$

This motivic zeta function was introduced by Kapranov in \cite{Ka}, where
he proves that $\zeta _X(t)$ is rational if $\dim X\leq 1$. He also says that
it is natural to expect rationality of $\zeta _X(t)$ for any variety $X$.

This conjecture of Kapranov was disproved in \cite{LaLu1} and \cite{LaLu2},
where we show that the motivic zeta function of a surface $X$ is
rational if and only if $X$ has Kodaira dimension $-\infty$ (for $\bbF =\bbC$).
The proof of this uses a ring homomorphism $K_0(\cV _{\bbC})\to \cH$ to a
field $\cH$ which factors through the quotient $K_0(\cV _{\bbC})/{\bbL}$, where
$\bbL =[\bbA ^1]$. Hence the question of rationality of the motivic
zeta function in the localized ring $K_0[\cV _{\bbF}][\bbL ^{-1}]$ is
still open.

In the paper \cite{DeLoe} the authors conjecture (Conjecture 7.5.1)
that $\zeta _X(t)$ {\it is} rational in
$K_0[\cV _{\bbF}][\bbL ^{-1}]$.

In this article we prove that the positive
solution to the Cut-and-paste problem implies that $\zeta _X(t)$ is
{\it not} rational in $K_0[\cV _{\bbF}][\bbL ^{-1}]$. This follows easily from
our results in \cite{LaLu1}.

We thank Ravi Vakil, whose beautiful recent lecture in Indiana University
on motivic Grothendieck ring prompted us to think again about the subject.

\section{Rationality of power series with coefficients in a ring}

Let $A$ be a commutative ring with 1. We recall and compare various notions
of rationality of power series with coefficients in $A$.

\begin{defi} \label{def1} A power series $f(t)\in A[[t]]$ is {\bf globally rational} if and
only if there exist polynomials $g(t), h(t)\in A[t]$ such that $f(t)$ is the unique
solution of $g(t)x = h(t)$.
\end{defi}

\begin{defi}  \label{def2} A power series $f(t)=\sum _{i=0}^\infty a_it^i \in A[[t]]$ is {\bf determinantally
rational} if and only if there exist integers $m$ and $n$ such that
$$\det \left( \begin{array}{cccc}
a_i & a_{i+1} & ... & a_{i+m} \\
a_{i+1} & a_{i+2} & ... & a_{i+m+1}\\
\vdots & \vdots & \ddots & \vdots \\
a_{i+m} & a_{i+m+1} & ... & a_{i+2m}
\end{array} \right)=0$$
for all $i>n$.
\end{defi}

It is classical that the Definition \ref{def1} is equivalent to Definition
\ref{def2} if $A$ is a field.

\begin{defi} \label{def3} A power series $f(t)\in A[[t]]$ is {\bf pointwise rational}
if and only if for all homomorphisms $\Phi $ from $A$ to a field, $\Phi (f)$ is
rational by either of the two previous definitions.
\end{defi}

These definitions are related by the following proposition \cite{LaLu2}, Prop. 2.4:

\begin{prop} Any globally rational power series is determinantally rational,
and any determinantally rational power series is pointwise rational. Neither
converse holds for a general coefficient ring $A$. All three conditions
are equivalent when $A$ is an integral domain.
\end{prop}

It is known that the ring $K_0[\cV _{\bbF}]$ has zero divisors \cite{Po}.

\section{Cut-and-paste problem and rationaly of $\zeta _X(t)$}

The following theorem was proved in \cite{LaLu2}, Thm. 7.6 and Cor. 3.8:

\begin{theo} Let $X$ be a complex surface of
Kodaira dimension $\geq 0$. Then the zeta function $\zeta _X(t)\in K_0[\cV _{\bbC}][[t]]$ is not pointwise rational.
\end{theo}

On the positive side it is relatively easy to prove the following
theorem \cite{LaLu2}, Thm. 3.9:

\begin{theo} If $X$ is a surface with the Kodaira dimension $-\infty$,
then the zeta function $\zeta _X(t)\in K_0[\cV _{\bbC}][[t]]$ is globally rational.
\end{theo}

Let $Y$ be a smooth projective variety of dimension $d$.
Recall that the polynomial
$$h_Y(s):=1+h^{1,0}(Y)s+h^{2,0}(Y)t^2+...+h^{d,0}(Y)s^d$$
is a birational invariant of $Y$ \cite{Hart}, Ch. II, Exercise 8.8. Here $h^{i,0}(Y)=\dim H^0(Y,\Omega ^i_Y)$.
Therefore we may (in characteristic zero) define $h_Z(t)$ for any variety $Z$, not necessarily smooth and projective, as
$$h_Z(s)=h_Y(s)$$
where $Y$ is any smooth projective model of $Z$.
The K\"{u}nneth formula for the Hodge structure on the cohomology of the constant sheaf $\bbC$ implies that $h_Y(s)$ is even a stable birational invariant of $Y$, i.e.
$$h_Y(s)=h_{Y\times {\bbP ^n}}(s)$$
The integer $P_g(Y):=h^{d,0}(Y)$ is the {\it geometric genus} of $Y$.

Here we prove the following theorem:

\begin{theo} Let $X$ be a complex surface with $P_g(X)\geq 2$. Assume that
the Cut-and-paste problem has a positive solution. Then the
zeta function $\zeta _X(t)\in K_0[\cV _{\bbC }][\bbL ^{-1}][[t]] $
is not determinantally rational.
\end{theo}

\begin{proof} Put $X^{(n)}:=\Symm ^nX$. If the zeta function
$\zeta _X(t)\in K_0[\cV _{\bbC }][\bbL ^{-1}][[t]]$ is determinantally rational
then there exist integers $n>0$ and $n_0>0$ such that for each $m>n_0$ the
determinant
\begin{equation}\label{det=0} \det \left( \begin{array}{cccc}
X^{(m)} & X^{(m+1)} & ... & X^{(m+n)} \\
X^{(m+1)} & X^{(m+2)} & ... & X^{(m+n+1)}\\
\vdots & \vdots & \ddots & \vdots \\
X^{(m+n)} & X^{(m+n+1)} & ... & X^{(m+2n)}
\end{array} \right)
\end{equation}
equals zero in the ring $K_0[\cV _{\bbC }][\bbL ^{-1}]$.
This determinant is the sum
\begin{equation}\label{explicit}
\sum _{\sigma \in S_{n+1}}\sign (\sigma)X^{(m-1+\sigma (1))}\times
X^{(m+\sigma (2))}\times ...
\times X^{(m+n-1+\sigma (n+1))}
\end{equation}
The assumption that the determinant is zero in $K_0[\cV _{\bbC }][\bbL ^{-1}]$
means that the quantity \ref{explicit} when multiplied by
some power $\bbL ^N$ is zero in $K_0[\cV _{\bbC }]$. Then the positive solution
to the Cut-and-paste problem implies that the various products in
the alternating sum \ref{explicit} when multiplied by $\bbL ^N$ become pairwise  birational (since all of
them have the same dimension). Note that the product
$$X^{(m)}\times X^{(m+2)}\times ...\times X^{(m+2n)}$$
appears exactly once in \ref{explicit}. Now we get a contradiction with the
following claim, which is proved on p. 11 in \cite{LaLu1}:

\medskip

\noindent{\it Claim.} For infinitely many $m>0$ the equality
$$P_g(X^{(m)}\times ...\times X^{(m+2n)})=
P_g(X^{(m-1+\sigma (1))}\times ...\times X^{(m+n-1+\sigma (n+1))})$$
implies that $\sigma =1$.
\end{proof}

\end{document}